\documentclass{amsart}
\usepackage{graphicx}
\usepackage{amssymb}
\newtheorem{thm}{Theorem}[section]
\newtheorem{cor}[thm]{Corollary}
\newtheorem{lemma}[thm]{Lemma}
\newtheorem{theorem}[thm]{Theorem}

\theoremstyle{definition}

\newtheorem{remark}[thm]{Remark}
\theoremstyle{question}

\theoremstyle{Conjecture}
\newtheorem{con}[thm]{Conjecture}
\numberwithin{equation}{section}

\begin{document}

\title[groups with given same order type]{A characterization of $A_5$ by its Same-order type}%
\author{L. Jafari Taghvasani and M. Zarrin}%

\address{Department of Mathematics, University of Kurdistan, P.O. Box: 416, Sanandaj, Iran}%
 \email{L.jafari@sci.uok.ac.ir and Zarrin@ipm.ir}
\begin{abstract}
Let $G$ be a group, define an equivalence relation $\thicksim$ as below:
\[\forall \ g, h \in G \ \ g\thicksim h \Longleftrightarrow |g|=|h|\]
the set of sizes of equivalence classes with respect to this relation is called the same-order type of $G$. Shen et al. (Monatsh. Math. 160 (2010), 337-341.), showed that $A_5$ is the only  group with the same-order type $\{1,15,20,24\}$. In this paper, among other things, we prove that a nonabelian simple group $G$  has same-order type $\{r,m,n,k\}$ if and only if $G\cong A_5$.
\end{abstract}
\maketitle

\section{\textbf{ Introduction and results}}

If $n$ is an integer, then we denote by $\pi(n)$ the set of all prime divisors of $n$.  Let $G$ be a group. Denote by $\pi(G)$ the set of primes $p$ such that $G$ contains an element of order
$p$, $\pi_e(G)$ the set of element orders of $G$ and $s_t$ be the number of elements of order $t$ in $G$ where $t\in \pi_e(G)$.

Let $nse(G) =
\{s_t | t\in \pi_e(G)\}$. In fact $nse(G)$ is the set of sizes of elements with the same order in $G$ (the influence
of $nse(G)$ on the structure of groups was studied by some authors (for instance see \cite{kka}, \cite{shen}). On the other hand, if we define an equivalence relation $\thicksim$ as below:
\[\forall \ g, h \in G \ \ g\thicksim h \Longleftrightarrow |g|=|h|,\]
then the set of sizes of equivalence classes with respect to this relation is (which called the same-order type of $G$ and denoted by $\alpha(G)$) equal to the set of sizes of elements with the same order in $G$. That is, $$nse(G)=\alpha(G).$$

 We say that a group $G$
is a $\alpha_n$-group if $|\alpha(G)|=n$ (or $|nse(G)|=n$). Clearly the only $\alpha_1$-groups are $1$ , $\Bbb{Z}_2$. In \cite{shen}, Shen showed that every $\alpha_2$-group ($\alpha_3$-group) is nilpotent (solvable, respectively). Furthermore he gave the structure of these groups and he raised the conjecture that if $G$ is a $\alpha_n$-group, then $|\pi(G)|\leq n$. Most recently, in \cite{jz}, the authors gave the partial answer to the Shen's conjecture and showed that his conjecture is true for the class of nilpotent groups.

In \cite{ssjm}, the authors showed that a group $G$ is isomorphic to $A_5$ if and only if $nse(G)=\alpha(A_5)=\{1,15,20,24\}$. Here we give  a new characterization of $A_5$ and show that $A_5$ is the only nonabelian simple group with the same-order type $\{r,m,n,k\}$ (note that as $s_e=1$ we may assume that $r=1$, where $e$ is the trivial element of $G$).

\begin{theorem}\label{t1}
Suppose that $G$ is a nonabelian simple group with same-order type $\{1,m,n,k\}$. Then $G\cong  A_5$.
\end{theorem}

 In fact, we show that $A_5$ is the only nonabelian simple $\alpha_4$-group. To prove our main result, we shall obtain a nice property of simple groups. In fact we show that for every nonabelian finite simple group  $G$, there exist two odd prime divisors $p$ and $q$ of the order of $G$ such that $s_p\neq s_q$ (see Lemma \ref{pq} and also Conjecture \ref{co}, below).

\section{\textbf{Proof of the Main Theorem}}

 First of all note that by Lemma $3$ of \cite{ssjm}, we can assume that $G$ is finite. To prove Theorem \ref{t1} we need the following lemmas. Denote by  $X_n$ is the set of all elements of order $n$ in a group $G$.

\begin{lemma}\label{d}
Let $G$ be a finite group and $k$ be a positive integer dividing $|G|$. Then $k\mid f(k)$ and $\phi (k)\mid s_k$, where $f(k)=|\{g \in G : g^k=1\}|$.
\end{lemma}
\begin{proof}
The proof is straightforward (also see \cite{f} and \cite{shen}).
\end{proof}

\begin{remark}
If $m\mid n$ then $\alpha{(C_m)}\subseteq \alpha{(C_n)}$, where $C_m$ and $C_n$ are cyclic groups of order $m$ and $n$, respectively.
\end{remark}
For any prime power $q$, we denote by $L_n(q)$ the projective special
linear group of degree $n$ over the finite field of size  $q$.
\begin{lemma}\label{2}
Let $G=L_2(q)$, where $q$ is a prime power of two or $q$ is a Fermat prime and or a Mersenne prime, then there exist $r, s\in \pi(G)\setminus \{2\}$ such that $s_r\neq s_s$.
\end{lemma}
\begin{proof}
Suppose, on the contrary, that $s_p=s_r$ for every two odd prime divisor $p$ and $r$ of $|G|$. It follows that $$s_{p_ip_j} \neq 0\text{ for any two odd prime~~~} p_i, p_j\in \pi(G).$$

 Let $q=t^m$ for some prime $t$ and a natural number $m$. It is well-known that the group $G=L_2(q)$ has a partition $\{P^g, U^g, S^g \mid g\in G\}$, where $U$ is cyclic subgroup of order $\frac{(t^m-1)}{(2,t^m-1)}$, $S$ is a cyclic subgroup of order $\frac{t^m+1}{(2,t^m-1)}$ and $P$ is a Sylow $t$-subgroup of $G$.

Suppose that $q$ is a prime power of two, then $|U^g|=q-1$ and $|S^g|=q+1$ for each $g\in G$. As $q-1$ and $q+1$ are two consecutive odd natural numbers, therefore there exists a prime odd number $p_1\in \pi(q-1)\setminus \pi(q+1)$ and also, as mentioned above, there exists $x\in X_{p_1p_i}$ for each $p_i \in \pi(S)=\pi(q+1)$. As $p_1\nmid q+1$ so $x \notin S^g$ for any $g\in G$ and so $x\in U^g$ for some $g\in G$ (note that every element of odd order of $G$ belongs to  $U^g$ or $S^g$ for some $g\in G$). Hence $\pi(q+1)\subsetneq \pi(q-1)$ which is a contradiction.

  If $q$ is a  Mersenne prime, then $t^m+1$ is a power of two and $\frac{(t^m-1)}{(2,t^m-1)}=2^x-1$ and $\frac{t^m+1}{(2,t^m-1)}=2^y$ for some positive numbers $x$ and $y$. Hence the elements of odd order of $G$ lie in cyclic subgroups of order $2^x-1$ (the number of such subgroups is $l=\frac{(t^m+1)t^m}{2}$) or Sylow $t$-subgroup of $G$. Now because of $(\frac{(t^m-1)}{(2,t^m-1)}, t^m)=1$. It follows that, for any $r,s \in \pi(G)\setminus \{2,t\}$ we have $s_r=l\phi(r)=l(r-1)$ and $s_s=l(s-1)$ and hence $s_r\neq s_s$, where $\phi(n)$ is the Euler totient function.

  If $q$ is a  Fermat prime, then a similar proof to that of Mersenne prime, gives the result.
\end{proof}
\begin{lemma}\label{4}
Let $G=Sz(q)$ be the Suzuki group, then there exist $r, s\in \pi(G)\setminus \{2\}$ such that $s_r\neq s_s$.
\end{lemma}
\begin{proof}
Suppose, on the contrary, that $s_p=s_r$ for every two odd prime divisor $p$ and $r$ of $|G|$. Therefore $s_{p_ip_j} \neq 0$ for every two odd prime $p_i, p_j\in \pi(G)$.

 The Suzuki group $G=Sz(q)$ ($q=2^{2m+1}$, $m> 0$) has a partition $$\{A^x, B^x, C^x, F^x | x\in G\},$$ where $A$, $B$, $C$ are cyclic subgroups and $F$ is a Sylow $2$-subgroup of $G$. Also $|F|=q^2$, $|A|=q-1$, $|B|=q-2r+1$ and $|C|=q+2r+1$, where $r=\sqrt{\frac{q}{2}}$. In fact $|A|=2^{2m+1}$, $|B|=2^{2m+1}-2^{m+1}+1$ and $|C|=2^{2m+1}+2^{m+1}+1$. Because of  $|C|-|B|=2^{m+2}$ and $|B|$ and $|C|$ are odd, we get $(|B|, |C|)=1$. Now assume that $x\in X_{p_ip_j}$, where  $p_i\in \pi(B)$ and $p_j\in \pi(C)$. It follows that $x\in A$ and so  $\pi(G)\setminus \{2\}\subseteq \pi (A)$. On the other hand, $|G|=q^2(q-1)(q^2+1)$, so $\pi(q^2+1)\subseteq \pi(q-1)$, a contradiction (since otherwise $\pi(q^2+1)\subseteq \pi(q^2-1)$ which is impossible, as $q^2-1$ and $q^2+1$ are two consecutive odd numbers).
\end{proof}

Let $p$ be a prime. A group $G$ is called a $C_{p,p}$-group if and only if $p\in \pi(G)$ and and the centralizers of its elements of order $p$ in $G$ are $p$-groups.

\begin{lemma}$($see \cite{suzuki}$)$\label{0}
If $G$ is a finite nonabelian simple $C_{2,2}$-group, then $G$ is isomorphic to one of the following groups.
\begin{itemize}
\item[(a)]$A_5$, $A_6$, $L_3(4)$;
\item[(b)] $L_2(q)$ where $q$ is a Fermat prime, a Mersenne prime or a prime power of $2$;
\item[(c)] Sz(q), where $q$ is odd prime power of $2$.
\end{itemize}
\end{lemma}
\begin{lemma} $($\cite{prime}, Lemma 9$)$ \label{1}
If $G\neq A_{10}$ is a finite simple group and $\Gamma (G)$ is connected, then there exist three primes $r, s, t\in \pi(G)$ such that $\{rs, tr, ts\}\cap \pi_e(G)=\emptyset$.
\end{lemma}
We obtain a interesting property of simple groups.
\begin{lemma}\label{pq}
Let $G$ be a nonabelian finite simple group. Then there exist two odd prime divisors $p$ and $q$ of the order of $G$ such that $s_p\neq s_q$.
\end{lemma}
\begin{proof}
Put $\pi (G) =\{2, p_1, p_2 , \dots , p_t\}$. Suppose, to the contrary, that $s_{p_1}=s_{p_2}= \dots =s_{p_t}=n$. Since
\begin{align*}
p_i\mid f(p_i p_j)=1+s_{p_j}+s_{p_i}+s_{p_i p_j}=1+2n+s_{p_ip_j}
\end{align*}
and $p_i\mid 1+s_{p_i}$, so $p_i\mid s_{p_i p_j}-1$ and $s_{p_i p_j}=k\not\in \{0, n\}$, for all $i, j \in \{1, 2, \dots, t\}$ and $i\neq j$. Now we have two cases :

 Case (1).\;If $s_{2p_i}=0$, for all $1\leq i \leq n$, then the centralizer of any element of order two is a $2$-group. That is, $G$ is a $C_{2,2}$-group. Now according to Lemmas \ref{0}, \ref{4} and \ref{2}, it is enough to consider the following groups: $G=A_6$ or $L_3(4)$. But in this case it is easy to see (by the computational
group theory system GAP) that for any two prime divisors $p$ and $q$ of the order of $G$,  $s_p\neq s_q$.

Case (2).\;If there exists $i\in \{1, 2, \dots, n\}$, such that $s_{2p_i}\neq 0$, then the prime graph of $G$ is connected. So by Lemma \ref{1}, there exist three primes $r, s, t\in \pi(G)$ such that $\{rs, tr, ts\}\cap \pi_e(G)=\emptyset$, a contradiction (as for all $i , j$ we have $s_{p_ip_j}=k\neq0$, by hypothesis).
\end{proof}

\begin{cor}
Let $G$ be a nonabelian finite simple group. Then there exist two odd prime divisors $p$ and $q$ of the order of $G$ such that $\{1,s_2,s_p,s_q\}\subseteq \alpha(G)=nse(G)$. In fact, every  nonabelian finite simple group is a $\alpha_n$-group with $n\geq 4$ {\rm (}note that $s_2\neq 1$, otherwise the center of $G$, $Z(G)\neq 1${\rm )}.
\end{cor}

\begin{lemma}\label{pi}
Assume that $G$ is a group such that $\alpha(G)=\{1,s_r, s_p ,s_q\}$. Then
\item[(i)]
If $r=2$,  then $\pi(G) =\{2,p,q\}$;
\item[(ii)]
If $|G|$ is an odd order, then $\pi(G) =\{p,q,r\}$.
\end{lemma}
\begin{proof}
Part (i).\;Since $s_2\neq 1$ we may assume that $s_2=n> 1$. Suppose, for a contradiction, that there exists $t\in \pi(G)\setminus\{2,p,q\}$. It follows that $s_p=m$, $s_q=s_t=k$ for some $m, k\in \mathbb{N}$.

By Lemma \ref{d}, $2\mid 1+s_2=1+n$, so $n$ is odd. Also, again by Lemma \ref{d}, one can conclude that $s_{2p}$, $s_{2q}$, $s_{2t}$, $s_{pq}$, $s_{pt}$, $s_{qt}$, $s_{2pq}$, $s_{2pt}$ and $s_{pqt}$ are even. Therefore
\begin{align*}
s_{2p}, s_{2q}, s_{2t}, s_{pq}, s_{pt}, s_{qt}, s_{2pq}, s_{2pt}, s_{pqt} \in \{0,m,k\}.
\end{align*}
On the other hand, $t\mid f(qt)=1+2k+s_{qt}$ and so $s_{qt}=m$. Also since $p\mid s_{qt}+s_{pqt}$, so $s_{pqt}=k$. Similarly, $s_{pq}=s_{pt}=m$. Now
\begin{align*}
t\mid f(pqt)&=1+s_p+s_q+s_t+s_{pt}+s_{pq}+s_{qt}+s_{pqt} \\
                  &=1+m+k+k+m+m+m+k \\
                  &=1+4m+3k
\end{align*}
and from $t\mid 1+2k+m$, $t\mid k+1$ and $t\mid m-1$, we get  $t\mid 2$, a contradiction.

Part (ii).\; Suppose on the contrary that there exists $t\in \pi(G)$ and $s_p=s_q=m$, $s_r=n$ and $s_t=k$. Since $p\mid f(pq)=1+2m+s_{pq}$, so $s_{pq} \in \{n,k\}$. Without loss of generality, let $s_{pq}=n$. Since $p\mid f(pq)=1+2m+n$, so $p\mid n-1$ and similarly $q\mid n-1$.

Now as $p\mid f(pr)=1+m+n+s_{pr}$, so $s_{pr}\in \{m,k\}$. Similarly $s_{qr}\in \{m, k\}$. Now we consider four cases:

\textbf{Case} (1)
Let $s_{pr}=s_{qr}=k$, then as $p\mid f(pr)$, so $p\mid k+1$. On the other hand, since $p\mid f(pt)=1+m+k+s_{pt}$, so $p\mid k+s_{pt}$, thus $s_{pt}=n$. Similarly, $s_{qt}=n$.
Now from relations $p\mid f(pqr)$ and $p\mid f(pt)$, we have $s_{pqr}=n$. Now as $r\mid f(pqr)=1+2m+3n+2k$ and $r\mid f(qr)=1+m+n+k$, so $r\mid 2$, a contradiction.

\textbf{Case}(2)
Let $s_{pr}=k$, $s_{qr}=m$. Now from relations $p\mid f(pr)$ and $p\mid f(pqr)$, this implies that $s_{pqr}=n$. Since $r\mid f(pr)=1+m+n+k$, $r\mid f(qr)=1+m+n+m$, $r\mid 1+n$. So $r\mid m$ and $r\mid k$. On the other hand, $r\mid f(pqr)=1+3m+3n+k$ and so $r\mid -2$, a contradiction.

\textbf{Case}(3)
Let $s_{pr}=m$ and $s_{qr}=k$. Since $r\mid f(pr)=1+2m+n$ and $r\mid f(qr)=1+m+n+k$, so $r\mid m$ and $r\mid k$. Also since $r\mid f(pqr)=1+3m+2n+k+s_{pqr}$, so $r\mid -1+s_{pqr}$ and so $s_{pqr}=n$, a contradiction.

\textbf{Case}(4)
Let $s_{pr}=s_{qr}=m$. Since $r\mid f(qr)=1+2m+n$, so $r\mid m$ and since $p\mid f(pqr)=1+4m+2n+s_{pqr}$, so $p\mid -1+s_{pqr}$, similarly $r\mid -1+s_{pqr}$, so $s_{pqr}=k$.

Now by a similar argument, we have:
\begin{align*}
s_{qrt}=s_{prt}=s_{pqt}=s_{pt}=s_{qt}=m, s_{rt}=s_{pqt}=n.
\end{align*}
Since $q\mid f(pqrt)=1+8m+4n+2k+s_{pqrt}$ and $q\mid f(qt)=1+2m+k$, so $q\mid k-1$ and since $q\mid n-1$ and $q\mid m+1$, so $q\mid -1+s_{pqrt}$ and $s_{pqrt}\in \{n,k\}$. Now if $s_{pqrt}=n$, then $r\mid s_{pqt}+s_{pqrt}=2n$, so $r\mid n$, which is a contradiction. If $s_{pqrt}=k$, then $t\mid s_{pqr}+s_{pqrt}=2k$, so $t\mid k$, which is a contradiction.
\end{proof}

We are now ready to conclude the proof of Theorem 1.1.\\

\noindent{\bf{Proof of Theorem1.1.}}
Since $G$ is simple so $s_2> 1$, let $s_2=m$. On the other hand, since $G$ is simple, so $|\pi(G)|\geqslant 3$.  By Lemma \ref{pq}, we can choose $\{2,p,q\} \subseteq \pi(G)$, such that $s_2 \neq s_p \neq s_q$ and so by lemma \ref{pi}, $\pi(G)=\{2, p ,q\}$. Now it is well-known that the nonabelian simple groups of order divisible by exact three primes are the following eight groups:
$L_2(q)$, where $q\in\{5, 7, 8, 9, 17\}$, $L_3(3)$, $U_3(3)$ and $U_4(2)$. Now by the computational
group theory system GAP, one can show that all mentioned groups are $\alpha_n$-group with $5\leq n$ except $A_5$.\\

In view of our main Theorem, we raise the following conjecture.

\begin{con}
 Let $S$ be a nonabelian simple $\alpha_n$-group and $G$ a $\alpha_n$-group such
that $|G| = |S|$. Then $G\cong S$.
\end{con}

We note that there are finite groups which are not characterizable even by $\alpha(G)$ and $|G|$. Thompson, in 1987, gave an example as follows: Let $G= (C_2\times C_2\times C_2\times C_2)\rtimes A_7$
and $H = L_3(4)\rtimes C_2$ be the maximal subgroups of $M_{23}$, the Mathieu group of degree 23. Then $\alpha(G)=\alpha(H)$ and $|G|=|H|$, but $G\cong H$. Also the condition $|G| = |S|$ is necessary in this conjecture. For example two groups $L_2(7)$ and $L_2(8)$ are $\alpha_5$-group and clearly $L_2(7)\not \simeq L_2(8)$.

Finally, in view of Lemma \ref{pq} and our computation together with the computational
group theory system GAP in investigating finite groups of small order suggests
that probably in every finite nonabelian simple group the following interesting property is satisfied (in fact, if $G$ is a finite group such that $s_p=s_q$ for some  prime divisors $p$ and $q$ of the order of $G$, then $G$ is not a simple group).

\begin{con}\label{co}
 Assume that $G$ be a finite nonabelian simple group. Then for any two prime divisors $p$ and $q$ of the order of $G$,  $s_p\neq s_q$.
\end{con}


\begin{thebibliography}{99}

\bibitem{f} G. Frobenius,  Verallgemeinerung des sylowschen satze. Berliner sitz, (1895) 981-993.
\bibitem{jz} L. Jafari Taghvasani  and M. Zarrin, Shen's conjecture on groups with given same order type,
Int. J. Group Theory, to appear.
\bibitem{kka} M. Khatami, B. Khosravi, Z. Akhlaghi, A new characterization for some linear groups.
Monatsh. Math. \textbf{163} (2011), 39-50.
\bibitem{prime}
M. S. Lucido and A. R Moghadamfar, Group with complete prime graph connected components,
J. Group theory, \textbf{31} (2004), 373-384.
\bibitem{sch} O. Yu. Schmidt, Groups all of whose subgroups are nilpotent, Mat. Sbornik \textbf{31} (1924),
366-372. (Russian).
\bibitem{shen} R. Shen, On groups with given same order type. Comm. Algebra \textbf{40} (2012), 2140-2150.
\bibitem{ssjm} R. Shen, C. Shao, Q. Jiang, W. Mazurov, A new characterization $A_5$, Monatsh. Math. \textbf{160} (2010), 337-341.
 \bibitem{suzuki} M. Suzuki, Finite groups with nilpotent centralizers,
Trans Amer Math \textbf{99} (1961), 425-470.
\bibitem{w}  J.S. Williams, Prime graph components of finite groups. J. Algebra \textbf{69} (1981) 487-513.
\end{thebibliography}
\end{document}